\pdfoutput=1
\RequirePackage{ifpdf}
\ifpdf % We~are running pdfTeX in pdf mode
\documentclass[pdftex]{sigma}
\else
\documentclass{sigma}
\fi

\numberwithin{equation}{section}

\newtheorem*{Theorem*}{Theorem}

 { \theoremstyle{definition}

 }

\newcommand{\br}{\mathbf{r}} % bold r for vectors
\newcommand{\bx}{\mathbf{x}} % bold x for vectors
\newcommand{\by}{\mathbf{y}} % bold y for vectors
\newcommand{\bo}{\mathbf{0}} % bold 0 for vectors
\newcommand{\ud}{\mathrm{d}}

\usepackage{pgfplots}
\pgfplotsset{compat=1.16}
\usetikzlibrary {arrows.meta}

\begin{document}
%\allowdisplaybreaks

\renewcommand{\thefootnote}{}

\newcommand{\arXivNumber}{2501.05746}

\renewcommand{\PaperNumber}{019}

\FirstPageHeading

\ShortArticleName{A Minimum Property for Cuboidal Lattice Sums}

\ArticleName{A Minimum Property for Cuboidal Lattice Sums\footnote{This paper is a~contribution to the Special Issue on Basic Hypergeometric Series Associated with Root Systems and Applications in honor of Stephen C.~Milne's 75th birthday. The~full collection is available at \href{https://www.emis.de/journals/SIGMA/Milne.html}{https://www.emis.de/journals/SIGMA/Milne.html}}}

\Author{Shaun COOPER~$^{\rm a}$ and Peter SCHWERDTFEGER~$^{\rm b}$}

\AuthorNameForHeading{S.~Cooper and P.~Schwerdtfeger}

\Address{$^{\rm a)}$~School of Mathematical and Computational Sciences, Massey University Albany,\\
\hphantom{$^{\rm a)}$}~Private Bag 102904, North Shore Mail Centre, Auckland 0745, New Zealand}
\EmailD{\href{mailto:s.cooper@massey.ac.nz}{s.cooper@massey.ac.nz}}

\Address{$^{\rm b)}$~Centre for Theoretical Chemistry and Physics, The New Zealand Institute for Advanced Study\\
\hphantom{$^{\rm b)}$}~(NZIAS), Massey University Albany, Private Bag 102904, North Shore Mail Centre,\\
\hphantom{$^{\rm b)}$}~Auckland 0745, New Zealand}
\EmailD{\href{mailto:peter.schwerdtfeger@gmail.com}{peter.schwerdtfeger@gmail.com}}

\ArticleDates{Received January 13, 2025, in final form March 16, 2025; Published online March 24, 2025}

\Abstract{We analyse a family of lattices considered by Conway and Sloane and show that the corresponding Epstein zeta function attains a local minimum for the body-centred cubic lattice.}


\Classification{11E45; 11H31}

\begin{flushright}
\begin{minipage}{58mm}
\it To Stephen Milne on the occasion\\ of his 75th birthday
\end{minipage}
\end{flushright}

\renewcommand{\thefootnote}{\arabic{footnote}}
\setcounter{footnote}{0}

\section{Introduction}
A lattice sum is an expression of the form
\[
\sum_{\bx} F(\bx),
\]
where the vector $\bx$ ranges over a $d$ dimensional lattice. In the case when
\[
F(\bx)=\begin{cases}
\dfrac{1}{|\bx|^s} & \text{if $\bx \neq \bo$}, \\
0 & \text{if $\bx = \bo$},
\end{cases}
\]
the lattice sum is a special case of Epstein's zeta function and has applications in number theory and theoretical chemistry, e.g., see~\cite{BorweinEtAl} for
comprehensive information or~\cite{StickyHardSphere} for some recent applications in chemistry.

Rankin~\cite{Rankin} showed that among two-dimensional lattice sums with discriminant~1, the Epstein zeta function is minimised by
the hexagonal lattice. Further results on minimising lattice sums have been given in~\cite{Betermin,cohn,edelsbrunner,Ennola,faulhuber,Fields,Montgomery,sarnak}.

In this work, we consider a one-parameter continuous family of lattices in three dimensions that includes the face-centred and body-centred cubic lattices, and show that
the Epstein zeta function for the family has a local minimum that is attained by the body-centred cubic lattice.

\section{Cuboidal lattices}
\label{S:2}
Following Conway and Sloane \cite{ConwaySloaneDual}, assume $u$ and $v$ are positive real numbers, let $A=u^2/v^2$, and consider the cuboidal lattice $\Lambda(u,v)$ generated by
\begin{equation}
\label{E:basis1}
\br_1 = (u,v,0),\qquad \br_2(u,0,v) \qquad \text{and}\qquad \br_3=(0,v,v),
\end{equation}
that is,
\begin{equation}
\label{E:basis2}
\Lambda(u,v) = \{c_1\br_1 + c_2\br_2 + c_3 \br_3 \mid c_1,c_2,c_3\in\mathbb{Z}\}.
\end{equation}
The most important cases, ordered by decreasing value of $A$, are
\begin{enumerate}\itemsep=0pt
\item[(1)] $A=1$: the face-centred cubic (fcc) lattice,
\item[(2)] $A=1/\sqrt2$: the mean centred-cuboidal (mcc) lattice,
\item[(3)] $A=1/2$: the body-centred cubic (bcc) lattice, and
\item[(4)] $A=1/3$: the axial centred-cuboidal (acc) lattice.
\end{enumerate}
Thus, $\Lambda(u,v)$ is a continuous family that includes the fcc and bcc lattices which are well known to chemistry students.
The other two lattices are less well known. The mcc lattice is the unique densest three-dimensional isodual lattice~\cite[Theorem 2]{ConwaySloaneDual}, while the
acc lattice is the least dense lattice with kissing number\footnote{See \cite[p. 21]{ConwaySloaneBook}.}~$10$,
e.g., see~\cite{ConwaySloaneDual,Fields2,Patterson}.

In the remainder of this section, we do some standard calculations to find the minimum norm and packing density for the lattice $\Lambda(u,v)$, and use them to define
the normalised lattice and associated Epstein zeta function.
If
\[
\bx = c_1\br_1 + c_2\br_2 + c_3\br_3\in\Lambda(u,v),
\]
then
\begin{align*}
\|\bx\|^2 = (c_1\br_1 + c_2\br_2 + c_3\br_3)\cdot(c_1\br_1 + c_2\br_2 + c_3\br_3)
 = (c_1,c_2,c_3) G (c_1,c_2,c_3)^{\sf T},
\end{align*}
where $G$ is the Gram matrix whose $(i,j)$ entry is the dot product $\br_i \cdot \br_j$, that is,
\begin{equation}
\label{E:Gram}
G=\begin{bmatrix}
 u^2+v^2 & u^2 & v^2 \\
 u^2 & u^2+v^2 & v^2 \\
 v^2 & v^2 & 2v^2
\end{bmatrix}
=v^2 \begin{bmatrix}
 A+1 & A & 1 \\
 A & A+1 & 1 \\
 1 & 1 & 2
\end{bmatrix}.
\end{equation}
The minimum norm of a lattice $\Lambda$ is \cite[p.~42]{ConwaySloaneBook}
\[
\mu = \text{min} \bigl\{ \| \bx -\by\|^2 \mid \bx,\by \in \Lambda, \; \bx \neq \by\bigr\} =
\text{min} \bigl\{ \| \bx \|^2 \mid \bx \in \Lambda, \; \bx \neq \bo \bigr\}.
\]
For the lattice $\Lambda(u,v)$, we have
\begin{align*}
\mu
&{}= \min \bigl\{ \| \bx \|^2 \mid \bx \in \Lambda(u,v), \; \bx \neq \bo \bigr\} \nonumber \\
&{}= \min \bigl\{ (c_1,c_2,c_3) G (c_1,c_2,c_3)^{\sf T}\mid c_1,c_2,c_3\in\mathbb{Z},\; (c_1,c_2,c_3)\neq(0,0,0) \bigr\} \nonumber \\
&{}=\min_{(c_1,c_2,c_3)\in\mathbb{Z}^3 \atop (c_1,c_2,c_3)\neq \bo\;} v^2 \bigl((A+1)c_1^2+(A+1)c_2^2+2c_3^2+2Ac_1c_2+2c_1c_3+2c_2c_3\bigr) \nonumber \\
&{}=\min_{(c_1,c_2,c_3)\in\mathbb{Z}^3 \atop (c_1,c_2,c_3)\neq \bo\;} v^2 \bigl( A(c_1+c_2)^2+(c_2+c_3)^2+(c_1+c_3)^2\bigr) \nonumber
\end{align*}
and it follows that
\begin{align}
\mu
&= \begin{cases}
4Av^2 & \text{if ~$0<A<1/3$,} \\
(A+1)v^2 & \text{if ~$1/3\leq A \leq 1$,} \\
2v^2 & \text{if ~$A>1$.}
\end{cases} \label{E:mu}
\end{align}
The packing radius of a lattice is defined by $p = \frac12 \sqrt{\mu}$ and the packing density is~\cite[pp.~1--7]{ConwaySloaneBook}
\[
\Delta = \frac{\text{volume of one sphere of radius $p$}}{(\det G)^{1/2}}.
\]
From~\eqref{E:Gram}, we have $\det G = 4Av^6$, hence the packing density of $\Lambda(u,v)$ is given by
\begin{equation}
\Delta =
\begin{cases}
\dfrac{2\pi A}{3} & \text{if~$0<A<1/3$,} \vspace{1.5mm}\\
\dfrac{\pi}{12}\sqrt{\dfrac{(A+1)^3}{A}} & \text{if~$1/3\leq A \leq 1$,} \vspace{1.5mm}\\
\dfrac{\pi}{6}\sqrt{\frac{2}{A}} & \text{if~$A>1$}.
\end{cases}
\label{E:density}
\end{equation}
The maximum packing
density for this family of lattices is~$\frac{\pi\sqrt{2}}{6} \approx 74\%$ for the fcc lattice at~${A=1}$. This was conjectured by Kepler and proved by Hales~\cite{Hales} to be maximal among all packings in three dimensions.
Since
\[
\frac{\ud}{\ud A} \left(\frac{\pi}{12}\sqrt{\frac{(A+1)^3}{A}}\right) = \frac{\pi}{24}\left(\frac{A+1}{A^3}\right)^{1/2}(2A-1),
\]
the packing density has a local minimum at $A=1/2$, corresponding to bcc.
Some special values of the density, together with the kissing numbers, are given in Table~\ref{T:1}.
A~graph of the packing density as a function of~$A$ is shown in Figure~\ref{F:graph}.

\begin{table}[h]\centering
\caption{Packing density and kissing number as functions of $A$.
Here $(a,b)$ means $\left\{A \mid a<A<b\right\}$.}

\vspace{1mm}

%\label{T:count}
\renewcommand{\arraystretch}{1.4}
\begin{tabular}{|c|c|c|c|c|c|c|c|c|c|}
\hline
$A$ & $(0,\frac13)$ & $\frac13$ & $(\frac13,\frac12)$ & $\frac12$ & $(\frac12,\frac1{\sqrt2})$ & $\frac1{\sqrt2}$ & $(\frac1{\sqrt2},1)$ & 1 & $(1,\infty)$ \\[2pt]
\hline
name & & acc & & bcc & & mcc & & fcc & \\
\hline
\rule{0pt}{4ex}
 density & $\frac{2\pi A}{3}$ & $\frac{2\pi}{9}$ & $\frac{\pi}{12}\sqrt{\frac{(A+1)^3}{A}}$ & $\frac{\pi\sqrt{3}}{8}$ & \multicolumn{3}{c|}{$\frac{\pi}{12}\sqrt{\frac{(A+1)^3}{A}}$} & $\frac{\pi\sqrt{2}}{6}$ & $\frac{\pi}{6}\sqrt{\frac{2}{A}}$ \\[5pt]
\hline
kiss. no. & 2 & 10 & \multicolumn{5}{c|}{8} & 12 & 4 \\
\hline
\end{tabular}
\label{T:1}
\end{table}

\begin{figure}[t]
\centering
\begin{tikzpicture}[thin,scale=5]

 \draw [black,thick,domain=0:0.334] plot ({\x}, {2*3.142*\x/3});
 \draw [black,thick,domain=0.334:1] plot ({\x}, {3.142*(\x+1)^(3/2)/(12*\x^(1/2))});
 \draw [black,thick,domain=1:1.5] plot ({\x}, {3.142/6*(2/\x)^(1/2)});
 \path[draw][-Stealth] (0,0) -- (1.5,0);
 \path[draw][-Stealth] (0,0) -- (0,1.1);

 \draw (1.5,0) node[below right] {$A$};
 \draw (0,1.1) node[left] {$\Delta$};

 \draw (1,-0.02) node[below] {$1$};
 \draw (1/2,0) node[below] {$\frac12$};
 \draw (1/3,0) node[below] {$\frac13$};

 \draw (1,-0.22) node[above] {fcc};
 \draw (1/2,-0.22) node[above] {bcc};
 \draw (1/3,-0.22) node[above] {acc};

 \draw (-0.12,0.82) node[left] {$\pi\sqrt{2}/6 \approx 0.740$};
 \draw (-0.12,0.72) node[left] {$2\pi/9 \approx 0.698$};
 \draw (-0.12,0.62) node[left] {$\pi\sqrt{3}/8 \approx 0.680$};

 \draw[dashed] (1,0) -- (1,3.1416*1.41421/6);
 \draw[dashed] (0,3.1416*1.41421/6) -- (1,3.1416*1.41421/6);
 \draw[dashed] (0,3.14159*1.732/8) -- (1/2,3.14159*1.732/8);
 \draw[dashed] (1/2,0) -- (1/2,3.14159*1.732/8);
 \draw[dashed] (0,2*3.14159/9) -- (1/3,2*3.14159/9);
 \draw[dashed] (1/3,0) -- (1/3,2*3.14159/9);

 \path[draw][dashed, very thin] (-0.12,0.79) -- (0,0.740);
 \path[draw][dashed, very thin] (-0.12,0.72) -- (0,0.698);
 \path[draw][dashed, very thin] (-0.12,0.64) -- (0,0.680);

\fill[black] (1,3.1416*1.41421/6) circle(0.01);
\fill[black] (1/2,3.14159*1.732/8) circle(0.01);
\fill[black] (1/3,2*3.14159/9) circle(0.01);
\fill[black] (0,0) circle(0.01);

\end{tikzpicture}
\caption{Graph of $\Delta$ as a function of $A$ given by~\eqref{E:density}.}
\label{F:graph}
\end{figure}
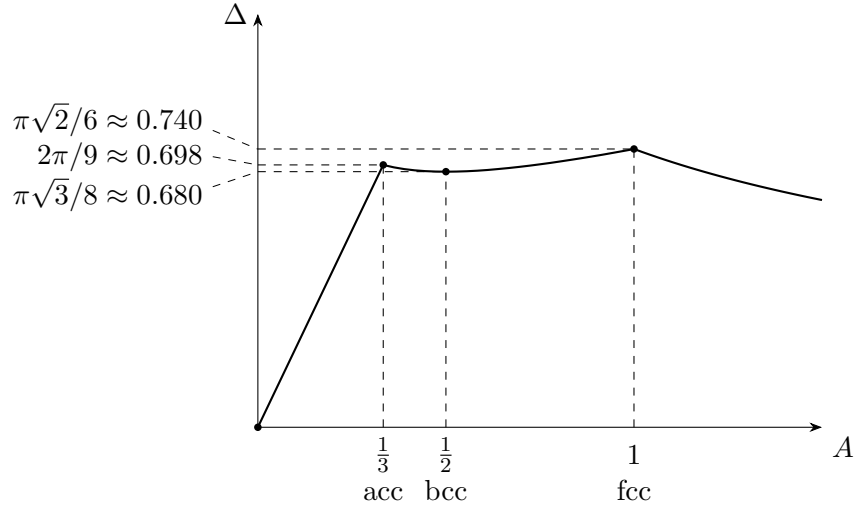

In order to study physical properties of the lattice, we normalise so that the minimum distance is~$\mu=1$.
Accordingly, from~\eqref{E:mu} let
\begin{equation}
v = \begin{cases}
\dfrac{1}{2\sqrt{A}} & \text{if~$0<A<1/3$,} \vspace{1mm}\\
\dfrac{1}{\sqrt{A+1}} & \text{if~$1/3\leq A \leq 1$,} \vspace{1mm}\\
\dfrac{1}{\sqrt{2}} & \text{if~$A>1$.}
\end{cases}
\label{E:normal}
\end{equation}
The cases $0<A<1/3$ and $A>1$ are of less interest because the normalised lattice degenerates to a lattice congruent to $\mathbb{Z}$ or $\mathbb{Z}^2$
in the limits $A\rightarrow 0$ and $A\rightarrow\infty$, respectively. From now on, we assume $1/3 \leq A \leq 1$. From~\eqref{E:Gram} and~\eqref{E:normal}, the Gram matrix for the normalised lattice~is\looseness=1
\[
G=\frac{1}{A+1}\begin{bmatrix}
 A+1 & A & 1 \\
 A & A+1 & 1 \\
 1 & 1 & 2
\end{bmatrix}.
\]
The associated quadratic form is
\begin{align*}
g(A;i,j,k) &= (i,j,k) G (i,j,k)^{\sf T} \\
&=\frac{1}{A+1}\bigl((A+1)i^2+(A+1)j^2+2k^2+2Aij+2ik+2jk\bigr) \\
&=\frac{1}{A+1}\bigl(A(i+j)^2+(j+k)^2+(i+k)^2\bigr)
\end{align*}
and the corresponding Epstein zeta function is
\begin{equation}
\label{E:Epstein}
L(A;s) = {\sum_{i,j,k\in\mathbb{Z}}\!\!\!}^{\prime}\biggl(\frac{1}{g(A;i,j,k)}\biggr)^s = {\sum_{i,j,k\in\mathbb{Z}}\!\!\!}^{\prime} \biggl(\frac{A+1}{A(i+j)^2+(j+k)^2+(i+k)^2}\biggr)^s,
\end{equation}
where the primes indicate that the term corresponding to $(i,j,k)=(0,0,0)$ is omitted from the summations. The series~\eqref{E:Epstein} converges for $s>3/2$ and we
further assume ${1/3\leq A \leq 1}$.

\section{A minimum property}
Numerical calculations in~\cite{StickyHardSphere} suggest that for a fixed value of~$s>3/2$, the Epstein zeta function~\eqref{E:Epstein}
appears to have a local minimum value at $A=1/2$. This is confirmed by the following~result.
\begin{theorem}
\label{theorem3point1}
Suppose $s>3/2$.
The Epstein zeta function $L(A;s)$ defined by~\eqref{E:Epstein} satisfies
\[
 \frac{\partial}{\partial A} L(A;s) \biggr|_{A=1/2} =0
\qquad \text{and}\qquad \frac{\partial^2}{\partial A^2} L(A;s) \biggr|_{A=1/2} >0.
\]
\end{theorem}
\begin{proof}
By definition, we have
\[
L(A;s) = {\sum_{I,J,K\in\mathbb{Z}}\!\!\!}^{\prime} \biggl( \frac{1}{g(A;I,J,K)}\biggr)^s,
\]
where
\[
g(A;I,J,K) = \frac{1}{A+1}\bigl(A(I+J)^2+(J+K)^2+(I+K)^2\bigr).
\]
Now make the change of variables $(I,J,K) = (i-j,-k,j)$. This is a bijection since $(i,j,k) = (I+K,K,-J)$, and it follows that
\begin{align*}
L(A;s) &= {\sum_{i,j,k\in\mathbb{Z}}\!\!\!}^{\prime} \biggl( \frac{1}{g(A;i-j,-k,j)}\biggr)^s \\
&={\sum_{i,j,k\in\mathbb{Z}}\!\!\!}^{\prime}~ \frac{1}{ \bigl(i^2+j^2+k^2-2(ij+ik)\bigl(\frac{A}{A+1}\bigr)+2jk\bigl(\frac{A-1}{A+1}\bigr)\bigr)^s}.
\end{align*}
By direct calculation, the derivative is given by
\begin{equation}
\label{deriv0}
\frac{\partial}{\partial A} L(A;s) =\frac{2s}{(A+1)^2}
{\sum_{i,j,k\in\mathbb{Z}}\!\!\!}^{\prime}~ \frac{ij+ik-2jk}{ \bigl(i^2+j^2+k^2-2(ij+ik)\bigl(\frac{A}{A+1}\bigr)+2jk\bigl(\frac{A-1}{A+1}\bigr)\bigr)^{s+1}}.
\end{equation}
Setting $A=1/2$ gives
\begin{equation}
\label{deriv1}
\frac{\partial}{\partial A} L(A;s)\biggr|_{A=1/2}
 =\frac{8s}{9}\;
{\sum_{i,j,k\in\mathbb{Z}}\!\!\!}^{\prime}~ \frac{ij+ik-2jk}{ \bigl(i^2+j^2+k^2-\frac23(ij+ik+jk)\bigr)^{s+1}}.
\end{equation}
Switching $i$ and $j$ gives
\begin{equation}
\label{deriv2}
\frac{\partial}{\partial A} L(A;s)\bigr|_{A=1/2}
 =\frac{8s}{9}\;
{\sum_{i,j,k\in\mathbb{Z}}\!\!\!}^{\prime}~ \frac{ij+jk-2ik}{ \bigl(i^2+j^2+k^2-\frac23(ij+ik+jk)\bigr)^{s+1}},
\end{equation}
while switching $i$ and $k$ in~\eqref{deriv1} gives
\begin{equation}
\label{deriv3}
\frac{\partial}{\partial A} L(A;s)\biggr|_{A=1/2}
 =\frac{8s}{9}\;
{\sum_{i,j,k\in\mathbb{Z}}\!\!\!}^{\prime}~ \frac{jk+ik-2ij}{ \bigl(i^2+j^2+k^2-\frac23(ij+ik+jk)\bigr)^{s+1}}.
\end{equation}
On adding \eqref{deriv1}--\eqref{deriv3} and noting that
\[(ij+ik-2jk) + (ij+jk-2ik) + (jk+ik-2ij) = 0,\]
it follows that
\[
\frac{\partial}{\partial A} L(A;s)\biggr|_{A=1/2} = 0.
\]
Next, taking the derivative of~\eqref{deriv0} gives
\begin{align*}
\frac{\partial^2}{\partial A^2} L(A;s)
&=\frac{-4s}{(A+1)^3}
{\sum_{i,j,k\in\mathbb{Z}}\!\!\!}^{\prime}~ \frac{ij+ik-2jk}{ \bigl(i^2+j^2+k^2-2(ij+ik)\bigl(\frac{A}{A+1}\bigr)+2jk\bigl(\frac{A-1}{A+1}\bigr)\bigr)^{s+1}} \\
\nonumber
&\quad + \frac{4s(s+1)}{(A+1)^4}
{\sum_{i,j,k\in\mathbb{Z}}\!\!\!}^{\prime}~ \frac{(ij+ik-2jk)^2}{ \bigl(i^2+j^2+k^2-2(ij+ik)\bigl(\frac{A}{A+1}\bigr)+2jk\bigl(\frac{A-1}{A+1}\bigr)\bigr)^{s+2}}.
\end{align*}
When $A=1/2$ the first sum is zero by the calculations in the first part of the proof. Therefore,
\begin{align*}
\frac{\partial^2}{\partial A^2} L(A;s) \biggr|_{A=1/2}
&= \frac{64s(s+1)}{81}
{\sum_{i,j,k\in\mathbb{Z}}\!\!\!}^{\prime}~ \frac{(jk+ik-2ij)^2}{ \bigl(i^2+j^2+k^2-\frac23(ij+ik+jk)\bigr)^{s+2}}.
\end{align*}
The term $(jk+ik-2ij)^2$ in the numerator is non-negative and not always zero. The
denominator is always positive because the quadratic form is positive definite. It follows that
\begin{align*}
\frac{\partial^2}{\partial A^2} L(A;s) \biggr|_{A=1/2} >0
\end{align*}
as required.

The calculations above are valid provided term-by-term differentiation of the series is allowed. All of the series encountered above
converge absolutely and uniformly on compact subsets of the region ${\rm Re}(s)>3/2$. On restricting~$s$ to real values, the conclusion
about positivity is valid for $s>3/2$.
\end{proof}

A consequence of Theorem~\ref{theorem3point1} is that for any fixed value $s>3/2$, the lattice
sum $L(A;s)$ attains a local minimum when~\mbox{$A=1/2$.}
The values $s=6$ and $s=3$ are used in the classical Lennard--Jones 12-6 potential, so the values $s>3/2$ are sufficient for physical applications.
Some graphs of $y=L(A;s)$ for $1/3\leq A \leq 1$ for various values of $s$ are given in Figure~\ref{F:graph2}.

\begin{figure}[t]
\centering
\begin{tikzpicture}[font=\small]
\begin{axis}
%[
%		xlabel=Cost]
		%ylabel=Error]
	\addplot[smooth,color=black,width=5cm, height=5cm] coordinates { %% data for s=20
 (0.33333, 10.00118)
 (0.35000, 8.967912)
 (0.36667, 8.489110)
 (0.38333, 8.25768)
 (0.40000, 8.141609)
 (0.416667, 8.081563)
 (0.43333, 8.049812)
 (0.46667, 8.0241740187)
 (0.50000, 8.0190273087)
 (0.53333, 8.0223952018)
 (0.56667, 8.03150688333)
 (0.60000, 8.04671837476)
 (0.63333, 8.06996385681)
 (0.66667, 8.104501758353)
 (0.70000, 8.155130930928)
 (0.73333, 8.228667598687)
 (0.76667, 8.334653114417)
 (0.80000, 8.486327201300)
 (0.83333, 8.7019353759012)
 (0.86667, 9.0064670553065)
 (0.90000, 9.4339506868159)
 (0.93333, 10.0304677378558)
 (0.96667, 10.8580909158806)
 (1, 12.0000057289405)
	};	

 \addplot[smooth,color=black,width=5cm, height=5cm] coordinates { %% data for s=3
 (0.33333333333333333333, 13.040622444562912403)
 (0.36666666666666666667, 12.694606628106355723)
 (0.40000000000000000000, 12.474552160963319153)
 (0.43333333333333333333, 12.342313594377939027)
 (0.46666666666666666667, 12.273908616045218811)
 (0.50000000000000000000, 12.253667867292322831)
 (0.53333333333333333333, 12.271016396284908976)
 (0.56666666666666666667, 12.318618566153222303)
 (0.60000000000000000000, 12.391265259106870626)
 (0.63333333333333333333, 12.485182814889144332)
 (0.66666666666666666667, 12.597590639621692929)
 (0.70000000000000000000, 12.726410301466212736)
 (0.73333333333333333333, 12.870069615379811433)
 (0.76666666666666666667, 13.027367845722104532)
 (0.80000000000000000000, 13.197381154105495280)
 (0.83333333333333333333, 13.379395107866104248)
 (0.86666666666666666667, 13.572855732323528115)
 (0.90000000000000000000, 13.777333492168977263)
 (0.93333333333333333333, 13.992496431159469827)
 (0.96666666666666666667, 14.218089894126007104)
 (1., 14.453921043744471864)
 };

 \addplot[smooth,color=black,width=5cm, height=5cm] coordinates {	%%%% data for s=6
 (0.33333333333333333333, 10.4079223304717222)
 (0.36666666666666666667, 9.7688437104513075)
 (0.40000000000000000000, 9.4176354059832301)
 (0.43333333333333333333, 9.2290396186405849)
 (0.46666666666666666667, 9.13929450239744972)
 (0.50000000000000000000, 9.11418326807535893)
 (0.53333333333333333333, 9.13459977678913771)
 (0.56666666666666666667, 9.18959607989457434)
 (0.60000000000000000000, 9.27284557888031306)
 (0.63333333333333333333, 9.38075698958812637)
 (0.66666666666666666667, 9.51142662974143924)
 (0.70000000000000000000, 9.66403529468611753)
 (0.73333333333333333333, 9.83849012381200527)
 (0.76666666666666666667, 10.03520615867881889)
 (0.80000000000000000000, 10.25497003062246171)
 (0.83333333333333333333, 10.49885329059140115)
 (0.86666666666666666667, 10.768156509951265180)
 (0.90000000000000000000, 11.064372899961851251)
 (0.93333333333333333333, 11.389164578796850756)
 (0.96666666666666666667, 11.744347197808978519)
 (1., 12.131880196544579717)
 };

 \addplot[smooth,color=black,width=5cm, height=5cm] coordinates {
 (1/3,8) (1,8)};

 \addplot+[only marks,mark=*,mark options={scale=0.5, fill=white},text mark as node=true,color=black] coordinates {
 (1/3,8)
 (1,8)
};

 \addplot+[only marks, mark=*, mark options={scale=0.5,color=black}] coordinates {
 (1/3,10)
 (1,12)
};

	\end{axis}
\end{tikzpicture}
\caption{Graphs of $y=L(A;s)$ for $1/3\leq A \leq 1$ given by~\eqref{E:Epstein} for (from top to bottom)
 $s=3$, ${s=6}$, ${s=20}$ and $s=\infty$.
 In the limiting case $s\rightarrow\infty$ we have
 $L(A;\infty)=\;$kiss$(A)$ where kiss$(A)$ is the kissing number as a function of $A$ as given in Table~\ref{T:1}.}
\label{F:graph2}
\end{figure}
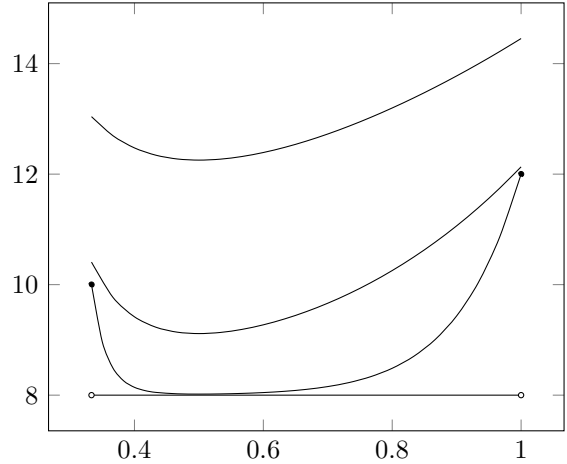

\appendix

\section[The limiting cases A -> 0 and A -> infty]{The limiting cases $\boldsymbol{A\rightarrow 0}$ and $\boldsymbol{A\rightarrow\infty}$}
In Section~\ref{S:2}, it was stated that the normalised lattices are congruent to $\mathbb{Z}$ or $\mathbb{Z}^2$
in the limits~${A\rightarrow 0}$ and $A\rightarrow\infty$, respectively.
Here we provide an analysis for the case $A\rightarrow\infty$. When $A>1$, by~\eqref{E:basis1}, \eqref{E:basis2} and~\eqref{E:normal} the normalised lattice is
\[
\Lambda= \biggl\{c_1 \biggl(\sqrt\frac{A}{2},\frac{1}{\sqrt2},0\biggr) + c_2 \biggl(\sqrt\frac{A}{2},0,\frac{1}{\sqrt2}\biggr)+c_3 \biggl(0,\frac{1}{\sqrt2},\frac{1}{\sqrt2}\biggr) \mid c_1,c_2,c_3\in\mathbb{Z}\biggr\}.
\]
In the limit $A\rightarrow\infty$, the only vectors that remain finite are those with $c_2=-c_1$, hence the limiting set is given by
\[
\Lambda_\infty := \lim_{A\rightarrow\infty}\Lambda = \biggl\{c_1 \biggl(0,\frac{1}{\sqrt2},\frac{-1}{\sqrt2}\biggr)+c_3 \biggl(0,\frac{1}{\sqrt2},\frac{1}{\sqrt2}\biggr) \mid c_1,c_3\in\mathbb{Z}\biggr\}.
\]
The function $\phi\colon \Lambda_\infty \rightarrow \mathbb{Z}^2$ given by
\[
\phi\biggl(c_1 \biggl(0,\frac{1}{\sqrt2},\frac{-1}{\sqrt2}\biggr)+c_3 \biggl(0,\frac{1}{\sqrt2},\frac{1}{\sqrt2}\biggr)\biggr) = (c_1,c_3)
\]
is an isometry, i.e., it preserves distance, hence the normalised lattice in the limiting case $A\rightarrow\infty$ is congruent to $\mathbb{Z}^2$.
The other limiting case $A\rightarrow0$ may be analysed in a similar way by considering the normalised lattice in the case $0<A<1/3$. We omit the details as they are similar.

%\bibliographystyle{sigma}
%\bibliography{example}

\pdfbookmark[1]{References}{ref}
\LastPageEnding

\end{document}